\documentclass[12pt,reqno]{amsart}
\usepackage{amsmath,amssymb,amsfonts,amscd,latexsym,amsthm,mathrsfs,verbatim,comment,cite}
\usepackage{hyperref}
\newtheorem{theorem}{Theorem}[section]

\numberwithin{equation}{section}

\begin{document}

\title[{Proof of a supercongruence through a $q$-microscope}]{Proof of a supercongruence conjectured by Sun through a $q$-microscope}


\author{Victor J. W. Guo}
\address{School of Mathematics and Statistics, Huaiyin Normal University, Huai'an 223300, Jiangsu, People's Republic of China}
\email{jwguo@hytc.edu.cn}

\thanks{The author was partially supported by the National Natural Science Foundation of China (grant 11771175).}


\begin{abstract}Recently, Z.-W. Sun made the following conjecture:
for any odd prime $p$ and odd integer $m$,
$$
\frac{1}{m^2{m-1\choose (m-1)/2}}\Bigg(\sum_{k=0}^{(pm-1)/2}\frac{{2k\choose k}}{8^k}
-\left(\frac{2}{p}\right)\sum_{k=0}^{(m-1)/2}\frac{{2k\choose k}}{8^k}\Bigg)
\equiv 0\pmod{p^2}.
$$
In this note, applying the ``creative microscoping" method, introduced by the author and Zudilin,
we confirm the above conjecture of Sun.

\end{abstract}

\maketitle

\section{Introduction}\label{sec1}
During the past decade, congruences and supercongruences have been studied by quite a few authors.
In 2011, Z.-W. Sun \cite[(1.6)]{Sun2} proved that, for any odd prime $p$,
\begin{align*}
\sum_{k=0}^{(p-1)/2}\frac{{2k\choose k}}{8^k}\equiv \left(\frac{2}{p}\right)
+\left(\frac{-2}{p}\right)\frac{p^2}{4}E_{p-3} \pmod{p^3},
\end{align*}
where $\big(\frac{\cdot}{\cdot}\big)$ denotes the Jacobi symbol and $E_n$ is the $n$-th Euler number.
Later, he \cite[(1.7)]{Sun4} further proved that
\begin{align}
\sum_{k=0}^{(p^r-1)/2}\frac{{2k\choose k}}{8^k}\equiv \left(\frac{2}{p^r}\right)
\pmod{p^2}. \label{eq:sun-2}
\end{align}
Recently, Z.-W. Sun \cite[Conjecture 4(ii)]{Sun-conj} also proposed the following conjecture:
for any odd prime $p$ and odd integer $m$,
\begin{align}
\frac{1}{m^2{m-1\choose (m-1)/2}}\Bigg(\sum_{k=0}^{(pm-1)/2}\frac{{2k\choose k}}{8^k}- \left(\frac{2}{p}\right)\sum_{k=0}^{(m-1)/2}\frac{{2k\choose k}}{8^k}\Bigg)
\equiv 0\pmod{p^2}, \label{eq:zp}
\end{align}
which is clearly a generalization of \eqref{eq:sun-2}.

In the past few years, $q$-analogues of congruences and supercongruences have caught the interests of lots of people
(see \cite{Gorodetsky,GG2,Guo-m3,Guo-Dwork,GL18,GPZ,GS2,GS3,GuoZu,GuoZu2,NP,Straub,Tauraso2}). In particular,
the author and Liu \cite{GL18} gave the following $q$-analogue of \eqref{eq:sun-2}: for odd $n>1$,
\begin{align}
\sum_{k=0}^{(n-1)/2}q^{k^2}\frac{(q;q^2)_k}{(q^4;q^4)_k}
\equiv (-q)^{(1-n^2)/8}\pmod{\Phi_n(q)^2}. \label{eq:qsun-1}
\end{align}
Here and in what follows, $(a;q)_n=(1-a)(1-aq)\cdots (1-aq^{n-1})$, $n=0,1,\ldots,$ or $n=\infty$,
is the {\em $q$-shifted factorial} and $\Phi_n(q)$ is the {\em $n$-th cyclotomic polynomial} in $q$ given by
\begin{equation*}
\Phi_n(q):=\prod_{\substack{1\leqslant k\leqslant n\\ \gcd(n,k)=1}}(q-\zeta^k),
\end{equation*}
where $\zeta$ is an $n$-th primitive root of unity. Moreover, Gu and the author \cite{GG2} gave some different $q$-analogues of \eqref{eq:sun-2}, such as
\begin{align}
\sum_{k=0}^{(n-1)/2}\frac{(q;q^2)_k q^{2k}}{(q^2;q^2)_k (-q;q^2)_k}
&\equiv \left(\frac{2}{n}\right) q^{2\lfloor (n+1)/4\rfloor^2}\pmod{\Phi_n(q)^2},  \label{eq:qsun-2}
\end{align}
where $\lfloor x\rfloor$ denotes the largest integer not exceeding $x$. In order to prove Z.-W. Sun's conjecture \eqref{eq:zp},
we need the following new $q$-analogue of \eqref{eq:sun-2}.

\begin{theorem}\label{thm:modsun}
Let $n>1$ be an odd integer. Then
\begin{align}
\sum_{k=0}^{(n-1)/2}\frac{(q;q^2)_k (-1;q^4)_k}{(-q;q^2)_k(q^4;q^4)_k}q^{2k} \equiv \bigg(\frac{2}{n}\bigg) \pmod{\Phi_n(q)^2}.  \label{eq:qsun-3}
\end{align}
\end{theorem}

Recall that the {\em $q$-integer} is defined by $[n]_q=1+q+\cdots+q^{n-1}$ and the {\em $q$-binomial coefficient} ${m\brack n}_q$ is defined as
\begin{align*}
{m\brack n}_q
=\begin{cases}
\dfrac{(q;q)_{m}}{(q;q)_n (q;q)_{m-n}}, &\text{if $0\leqslant n\leqslant m$},\\[10pt]
0, &\text{otherwise}.
\end{cases}
\end{align*}
Based on \eqref{eq:qsun-3}, we are able to give the following $q$-analogue of \eqref{eq:zp}.
\begin{theorem}\label{thm:modsun-zp}
Let $m$ and $n$ be positive odd integers with $n>1$. Then
\begin{align}
&\frac{1}{[m]_{q^n}^2{m-1\brack (m-1)/2}_{q^n}}\Bigg(\sum_{k=0}^{(mn-1)/2}\frac{(q;q^2)_k (-1;q^4)_k}{(-q;q^2)_k(q^4;q^4)_k}q^{2k} \notag \\[5pt]
&\qquad\qquad\qquad\qquad -\bigg(\frac{2}{n}\bigg)\sum_{k=0}^{(m-1)/2}\frac{(q^n;q^{2n})_k (-1;q^{4n})_k}{(-q^n;q^{2n})_k(q^{4n};q^{4n})_k}q^{2nk}\Bigg) \notag \\[5pt]
&\quad\equiv 0\pmod{\Phi_n(q)^2}.  \label{eq:qsun-4}
\end{align}
Moreover, the denominator of {\rm(}the reduced form of{\rm)} the left-hand side of \eqref{eq:qsun-4} is relatively prime to $\Phi_{n^j}(q)$ for any integer $j\geqslant 2$.
\end{theorem}

It is well known that $\Phi_n(1)=p$ if $n$ is a prime power $p^r$ ($r\geqslant 1$) and $\Phi_n(1)=1$ otherwise.
Moreover, the denominator of \eqref{eq:qsun-4} is no doubt a product of cyclotomic polynomials.
This immediately means that \eqref{eq:zp} follows from \eqref{eq:qsun-4} by letting $m\mapsto n$, $n\mapsto p$
and taking the limits as $q\to 1$.

We shall prove Theorem \ref{thm:modsun} in the next section. The proof of Theorem \ref{thm:modsun-zp} will be given
in Section 3 by using the method of ``creative microscoping" recently introduced by the author and Zudilin \cite{GuoZu}.
More precisely, we shall first give a generalization of Theorem \ref{thm:modsun-zp} with an extra parameter $a$,
and Theorem \ref{thm:modsun-zp} then follows from this generalization by taking $a\to 1$. We end this note with
some remarks on another similar conjecture of Z.-W. Sun in Section 4.

\section{Proof of Theorem \ref{thm:modsun}}
It is easy to check that
$$
(1-q^{n-2j+1})(1-q^{n+2j-1})+(1-q^{2j-1})^2q^{n-2j+1}=(1-q^n)^2
$$
and $1-q^n\equiv 0\pmod{\Phi_n(q)}$, and so
\begin{align*}
(1-q^{n-2j+1})(1-q^{n+2j-1})\equiv -(1-q^{2j-1})^2q^{n-2j+1}\pmod{\Phi_n(q)^2}.
\end{align*}
Thus, we have
\begin{align*}
(q^{1-n};q^2)_k(q^{1+n};q^2)_k
&=(-1)^k q^{k^2-nk}\prod_{j=1}^{k}(1-q^{n-2j+1})(1-q^{n+2j-1})\\[5pt]
&\equiv q^{k^2-nk}\prod_{j=1}^{k}(1-q^{2j-1})^2q^{n-2j+1} \\[5pt]
&=(q;q^2)_k^2 \pmod{\Phi_n(q)^2}.
\end{align*}
It follows that
\begin{align}
\sum_{k=0}^{(n-1)/2}\frac{(q;q^2)_k (-1;q^4)_k}{(-q;q^2)_k(q^4;q^4)_k}q^{2k}
&\equiv \sum_{k=0}^{(n-1)/2}\frac{(q^{1-n};q^2)_k(q^{1+n};q^2)_k (-1;q^4)_k}{(q;q^2)_k(-q;q^2)_k(q^4;q^4)_k}q^{2k} \notag\\[5pt]
&= \bigg(\frac{2}{n}\bigg)
 \pmod{\Phi_n(q)^2}. \label{eq:last}
\end{align}
Here the last step in \eqref{eq:last} follows from a terminating $q$-analogue of Whipple's $_3F_2$ sum \cite[Appendix (II.19)]{GR}:
\begin{align*}
&\sum_{k=0}^{n}\frac{(q^{-n};q)_k (q^{n+1};q)_k (c;q)_k (-c;q)_k }{(e;q)_k (c^2q/e;q)_k (q;q)_k(-q;q)_k  }q^k \\[5pt]
&\quad=\frac{(eq^{-n};q^2)_\infty (eq^{n+1};q^2)_\infty (c^2q^{1-n}/e;q^2)_\infty (c^2q^{n+2}/e;q^2)_\infty}
{(e;q)_\infty (c^2q/e;q)_\infty} q^{n(n+1)/2}
\end{align*}
with $n\mapsto\frac{n-1}{2}$, $q\mapsto q^2$, $c^2=-1$ and $e=q$.

\section{Proof of Theorem \ref{thm:modsun-zp}}
We first establish the following parametric generalization of Theorem \ref{thm:modsun-zp}.
\begin{theorem}\label{thm:modsun-zp-a}
Let $m$ and $n$ be positive odd integers with $n>1$. Then, modulo
\begin{align}
\prod_{j=0}^{(m-1)/2}(1-aq^{(2j+1)n})(a-q^{(2j+1)n}),  \label{eq:prod}
\end{align}
we have
\begin{align}
&\sum_{k=0}^{(mn-1)/2}\frac{(aq;q^2)_k(q/a;q^2)_k (-1;q^4)_k}{(q;q^2)_k(-q;q^2)_k(q^4;q^4)_k}q^{2k}  \notag\\[5pt]
&\quad\equiv\bigg(\frac{2}{n}\bigg)\sum_{k=0}^{(m-1)/2}\frac{(aq^n;q^{2n})_k(q^n/a;q^{2n})_k (-1;q^{4n})_k}{(q^n;q^{2n})_k(-q^n;q^{2n})_k(q^{4n};q^{4n})_k}q^{2nk}.
\label{eq:qsun-4a}
\end{align}
\end{theorem}
\begin{proof}
It suffices to prove that both sides of \eqref{eq:qsun-4a} are identical for $a=q^{-(2j+1)n}$ and $a=q^{(2j+1)n}$ with $j=0,1,\ldots,(m-1)/2$, i.e.,
\begin{align}
&\sum_{k=0}^{(mn-1)/2}\frac{(q^{1-(2j+1)n};q^2)_k(q^{1+(2j+1)n};q^2)_k (-1;q^4)_k}{(q;q^2)_k(-q;q^2)_k(q^4;q^4)_k}q^{2k}  \notag\\[5pt]
&\quad=\bigg(\frac{2}{n}\bigg)\sum_{k=0}^{(m-1)/2}\frac{(q^{-2jn};q^{2n})_k(q^{2-2jn};q^{2n})_k (-1;q^{4n})_k}{(q^n;q^{2n})_k(-q^n;q^{2n})_k(q^{4n};q^{4n})_k}q^{2nk}.
\label{eq:qsun-4aq}
\end{align}
Clearly, $(mn-1)/2\geqslant ((2j+1)n-1)/2$ for $0\leqslant j\leqslant (m-1)/2$,
and $(q^{1-(2j+1)n};q^2)_k=0$ for $k>((2j+1)n-1)/2$. By the identity in \eqref{eq:last}, the left-hand side of \eqref{eq:qsun-4aq} is equal to
$\big(\frac{2}{(2j+1)n}\big)$. Likewise, the right-hand side of \eqref{eq:qsun-4aq} is equal to
$$
\bigg(\frac{2}{n}\bigg)\bigg(\frac{2}{2j+1}\bigg)=\bigg(\frac{2}{(2j+1)n}\bigg),
$$
where $\big(\frac{2}{1}\big)$ is understood to be $1$.
This establishes the identity \eqref{eq:qsun-4aq}, and so the $q$-congruence \eqref{eq:qsun-4a} holds.
\end{proof}

Now we can prove Theorem \ref{thm:modsun-zp}.
\begin{proof}[Proof of Theorem {\rm\ref{thm:modsun-zp}}]
It is easy to see that
$$
q^N-1=\prod_{d|N}\Phi_d(q),
$$
and there are $\lfloor m/n^{j-1}\rfloor-\lfloor (m-1)/(2n^{j-1})\rfloor$
multiples of $n^{j-1}$ in the arithmetic progress $1,3,\ldots, m$ for any positive integer $j$.
Thus, the limit of \eqref{eq:prod} as $a\to 1$ has the factor
\begin{align*}
\prod_{j=1}^\infty\Phi_{n^j}(q)^{2\lfloor m/n^{j-1}\rfloor-2\lfloor (m-1)/(2n^{j-1})\rfloor}.
\end{align*}

On the other hand, the denominator of the left-hand side of \eqref{eq:qsun-4a} is divisible by that of the right-hand side of \eqref{eq:qsun-4a}.
The former is equal to $(q^2;q^2)_{mn-1}$ and its factor related to $\Phi_n(q),\Phi_{n^2}(q),\ldots$
is just
\begin{align*}
\prod_{j=1}^\infty\Phi_{n^j}(q)^{\lfloor (mn-1)/n^j\rfloor}.
\end{align*}
Moreover, writing $[m]_q=(q;q)_m/((1-q)(q;q)_{m-1})$, the $q$-binomial coefficient ${m-1\brack (m-1)/2}$ as
a product of cyclotomic polynomials (see, for example, \cite{CH}), and then using the fact $\Phi_{n^j}(q^n)=\Phi_{n^{j+1}}(q)$,
we know that the polynomial $[m]_{q^n}^2{m-1\brack (m-1)/2}_{q^n}$ only has the following factor
\begin{align*}
\prod_{j=2}^\infty\Phi_{n^{j}}(q)^{2\lfloor m/n^{j-1}\rfloor-\lfloor (m-1)/n^{j-1}\rfloor-2\lfloor (m-1)/(2n^{j-1})\rfloor}
\end{align*}
related to $\Phi_n(q),\Phi_{n^2}(q),\ldots.$

It is clear that
$$
2\lfloor m/n^{j-1}\rfloor-2\lfloor (m-1)/(2n^{j-1})\rfloor
-\lfloor (mn-1)/n^j\rfloor=2\quad\text{for $j=1$},
$$
and
$$
\lfloor (mn-1)/n^j\rfloor=\lfloor (mn-n)/n^j\rfloor =\lfloor (m-1)/n^{j-1}\rfloor \quad\text{for $j\geqslant 1$}.
$$
Therefore, letting $a\to 1$ in \eqref{eq:qsun-4a}, we see that
the $q$-congruence \eqref{eq:qsun-4} holds, and the denominator of the left-hand side of
\eqref{eq:qsun-4} is relatively prime to $\Phi_{n^j}(q)$ for $j\geqslant 2$, as desired.
\end{proof}

\section{Concluding remarks}
Z.-W. Sun \cite[Conjecture 4(ii)]{Sun-conj} also made the following conjecture:
for any odd prime $p$ and odd integer $m$,
\begin{align}
\frac{1}{m^2{m-1\choose (m-1)/2}}\left(\sum_{k=0}^{(pm-1)/2}\frac{{2k\choose k}}{16^k}
-\left(\frac{3}{p}\right)\sum_{k=0}^{(m-1)/2}\frac{{2k\choose k}}{16^k}\right)
\equiv 0\pmod{p^2}, \label{eq:zp-3}
\end{align}
of which the $m=1$ case was already proved by Sun \cite{Sun4} himself. Although Gu and the author \cite{GG2}
gave the following $q$-analogue of \eqref{eq:zp-3} for $m=1$: for odd $n>1$,
\begin{align}
\sum_{k=0}^{(n-1)/2}\frac{(q;q^2)_k q^{2k}}{(q^4;q^4)_k (-q;q^2)_k}
&\equiv \left(\frac{3}{n}\right)q^{(n^2-1)/12}  \pmod{\Phi_n(q)^2},   \label{eq:qSun-2}
\end{align}
we cannot utilize \eqref{eq:qSun-2} to give a $q$-analogue of \eqref{eq:zp-3} similar to Theorem \ref{thm:modsun-zp} because $(n^2-1)/12$
is not a linear function of $n$. Anyway, we believe that such a $q$-analogue of \eqref{eq:zp-3}
should exist, which is left to interested reader.


\begin{thebibliography}{99}

\bibitem{CH}W.Y.C. Chen and Q.-H. Hou, Factors of the Gaussian coefficients,
Discrete Math. 306 (2006), 1446--1449.

\bibitem{GR} G. Gasper and M. Rahman,
Basic hypergeometric series, second edition, Encyclopedia of
Mathematics and Its Applications 96, Cambridge University Press, Cambridge, 2004.

\bibitem{Gorodetsky}O. Gorodetsky, $q$-Congruences, with applications
to supercongruences and the cyclic sieving phenomenon, Int. J. Number Theory 15 (2019), 1919--1968.

\bibitem{GG2}C.-Y. Gu and V.J.W. Guo, $q$-Analogues of two supercongruences of Z.-W. Sun, Czechoslovak Math. J., to appear.

\bibitem{Guo-m3}V.J.W. Guo, Common $q$-analogues of some different supercongruences,
Results Math. 74 (2019), Art. 131.

\bibitem{Guo-Dwork}V.J.W. Guo, $q$-Analogues of Dwork-type supercongruences,
preprint, 2019; arXiv:1910.07551.

\bibitem{GL18}V.J.W. Guo and J.-C. Liu, $q$-Analogues of two Ramanujan-type formulas for $1/\pi$,
J. Difference Equ. Appl. 24 (2018), 1368--1373.

\bibitem{GPZ}V. J. W. Guo, H. Pan and Y. Zhang, The Rodriguez-Villegas type congruences for truncated $q$-hypergeometric functions,
J. Number Theory 174 (2017), 358--368.

\bibitem{GS2}V.J.W. Guo and M.J. Schlosser,
Some new $q$-congruences for truncated basic hypergeometric series:
even powers, Results Math. 75 (2020), Art. 1.

\bibitem{GS3}V.J.W. Guo and M.J. Schlosser, A family of $q$-hypergeometric congruences modulo the fourth power of a cyclotomic polynomial,
Isarel J. Math., to appear.

\bibitem{GuoZu}V.J.W. Guo and W. Zudilin, A $q$-microscope for supercongruences, Adv. Math. 346 (2019), 329--358.

\bibitem{GuoZu2}V.J.W. Guo and W. Zudilin, A common $q$-analogue of two supercongruences, preprint, 2019, arXiv:1910.10932.

\bibitem{NP}H.-X. Ni and H. Pan, Some symmetric $q$-congruences modulo the square of a cyclotomic polynomial,
J. Math. Anal. Appl. 481 (2020), Art. 123372.

\bibitem{Straub}A. Straub, Supercongruences for polynomial analogs of the Ap\'ery numbers,
Proc. Amer. Math. Soc. 147 (2019), 1023--1036.

\bibitem{Sun2}Z.-W. Sun, Super congruences and Euler numbers, Sci. China Math. 54 (2011), 2509--2535.

\bibitem{Sun4}Z.-W. Sun, Fibonacci numbers modulo cubes of primes, Taiwan. J. Math. 17 (2013), 1523--1543.

\bibitem{Sun-conj}Z.-W. Sun, Open conjectures on congruences, Nanjing Univ. J. Math. Biquarterly 36 (2019), no. 1, 1--99.

\bibitem{Tauraso2}R. Tauraso, $q$-Analogs of some congruences involving Catalan numbers, Adv. Appl. Math. 48 (2009), 603--614.

\end{thebibliography}
\end{document}